\documentclass[11pt]{amsart}
\usepackage[T1]{fontenc}
\usepackage[utf8]{inputenc}
\usepackage[english]{babel}
\usepackage{graphicx}
\usepackage{amsmath}
\usepackage{amssymb}
\usepackage{enumerate}
\usepackage[all,cmtip]{xy}
\usepackage{mathabx}
\usepackage{amsthm}
\usepackage{amsfonts}

\newtheorem{thm1}{Theorem}[section]
\newtheorem{theorem}[thm1]{Theorem}
\newtheorem{lemma}[thm1]{Lemma}
\newtheorem{corollary}[thm1]{Corollary}
\newtheorem{proposition}[thm1]{Proposition}

\newtheorem{thmx}{Theorem}

\RequirePackage{ifthen}
\RequirePackage{calc}
\newcounter{exampleendflag}
\newcommand{\exendhere}{
  \setcounter{exampleendflag}{0} 
  \ifmmode
    \eqno
    \ensuremath{\blacktriangle}
  \else
    \hspace{\stretch{1}}
    \ensuremath{\blacktriangle}
  \fi
}
\newenvironment{example}{
  \setcounter{exampleendflag}{1}
  \begin{exx}
}{
  \ifthenelse{\value{exampleendflag}=1}{\exendhere}{} 
  \end{exx}
}

\theoremstyle{definition}
\newtheorem{definition}[thm1]{Definition}

\theoremstyle{remark}
\newtheorem{remark}[thm1]{Remark}
\newtheorem{exx}[thm1]{Example}
\newtheorem{ex2}[thm1]{Example}

\newcommand{\fF}{\mathcal{F}}
\newcommand{\fR}{\mathcal{R}}
\newcommand{\sym}{\operatorname{Sym}}
\newcommand{\coker}{\operatorname{coker}}
\newcommand{\moda}{\mathbf{{Mod}}_A}
\newcommand{\funa}{\mathbf{{Fun}}_A}

\newcommand{\im}{{\operatorname{im}}}  
\renewcommand{\hom}{\operatorname{Hom}}
\newcommand{\id}{\operatorname{id}}
\newcommand{\coeq}{\operatorname{coeq}}

\renewcommand{\phi}{\varphi}
\renewcommand{\epsilon}{\varepsilon}

\numberwithin{equation}{section}
\author{Gustav S{\ae}d{\'e}n St{\aa}hl}
\address{Department of Mathematics, KTH Royal Institute of Technology, SE-100 44 Stockholm, Sweden}
\email{gss@math.kth.se}
\title{Rees algebras of modules and coherent functors}

\begin{document}
\begin{abstract}
We show that several properties of the theory of Rees algebras of modules become more transparent using the category of coherent functors rather than working directly with modules. In particular, we show that the Rees algebra is induced by a canonical map of coherent functors.
\end{abstract}

\thanks{The author is supported by the Swedish Research Council, grant number 2011-5599.}

\maketitle
\section*{Introduction}
In \cite{reesbud}, the authors give a definition of the Rees algebra of a finitely generated module over a noetherian ring. This definition was also studied in \cite{jag1}, where we showed that the Rees algebra $\fR(M)$ of a finitely generated module $M$ is equal to the image of a canonical map ${\sym(M)\to\Gamma(M^\ast)^\vee}$ from the symmetric algebra of $M$ to the graded dual of the algebra of divided powers of the dual of the module $M$. In this paper, we use coherent functors to obtain nice characterizations of properties of the Rees algebra that are not available in the category of modules. Two of these results are summarized in Theorems~\ref{thm:intro2} and~\ref{thm:introphi}. 

For any finitely generated module $M$ over a noetherian ring $A$, we consider the  functors ${t_M=M\otimes_A(-)}$ and $h^M=\hom_A(M,-)$. There is a canonical map $t_M\to h^{M^\ast}$, and we introduce the functor \[{r}_M=\im\bigl(t_M\to h^{M^\ast}\bigr).\]
\begin{thmx}\label{thm:intro2}
Let $A$ be a noetherian ring and let $M\to N$ be a homomorphism of finitely generated $A$-modules. If the induced morphism ${r}_M\to{r}_N$ is injective (resp. surjective), then $\fR(M)\to\fR(N)$ is injective (resp. surjective).
\end{thmx}

In particular, letting $N=F$ be a free module, a homomorphism $M\to F$ that induces an injection of functors ${r}_M\to{r}_F$ is a versal map in the terminology of \cite{reesbud}. Given such a map, the theorem implies that  ${\fR(M)\to\fR(F)=\sym(F)}$ is injective, recovering another result of \cite{reesbud}, namely that the Rees algebra of $M$ can be computed as the image of  the map $\sym(M)\to\sym(F)$.

Another result we obtain, that is not available in the category of modules, is the following.
\begin{thmx}\label{thm:introphi}
There is a functor $\Phi$ from the category of coherent functors to the category of finitely generated and graded algebras over $A$ such that ${\Phi(t_M)=\sym(M)}$ and ${\Phi(h^M)=\im\bigl(\sym(M^\ast)\to\Gamma(M)^\vee\bigr)}$. 
\end{thmx}
By the result of \cite{jag1}, this theorem shows that the Rees algebra of a module $M$ is equal to the image of the map given by applying $\Phi$ to the canonical map $t_M\to h^{M^\ast}$, that is, \[\fR(M)=\im\Bigl(\Phi\bigl(t_M\bigr)\to\Phi\bigl(h^{M^\ast}\bigr)\Bigr).\]

\noindent\textbf{Background.} 
The Rees algebra of an ideal is a fundamental object in algebraic geometry. As ideals are special cases of modules, it is reasonable to ask if the Rees algebra generalizes in some natural way to the case of modules. In \cite{reesbud}, such a generalization was presented by defining the Rees algebra of finitely generated modules over noetherian rings in terms of maps to free modules. In \cite{jag1}, we expanded on these ideas and found an intrinsic definition of the Rees algebra in terms of the algebra of divided powers. 

A seemingly unrelated topic is the theory of coherent functors introduced by Auslander \cite{auslander}, and also studied by Hartshorne \cite{coherenthart}. These objects have been used for describing various results, such as Schlessinger's approach to infinitesimal deformation theory \cite{schartin}, and Hall's reformulation of Artin's criterion for the algebraicity of a stack~\cite{2012arXiv1206.4182H}.

\par\vspace{\baselineskip}
\noindent\textbf{Structure of the paper.} We start in Section~\ref{sec:1} by reviewing some results of Rees algebras of modules. In particular, we state some results concerning the notion of versal maps. 

In Section~\ref{sec:torsionless} we discuss the torsionless quotient of a module $M$, which is defined as the image of the canonical map from $M$ to its double dual. We find that any versal map will factor through this quotient and show that it is connected to the Rees algebra of $M$.

Sections~\ref{sec:verfunc} and \ref{sec:tf} are focused on the study of relations between coherent functors, such as $h^M=\hom_A(M,-)$ and $t_M=M\otimes_A(-)$, and the Rees algebra. We give a characterization of versal maps in terms of coherent functors and find that their defining properties become more transparent in this setting. After generalizing the concept of the torsionless quotient of a module $M$ to a torsionless functor ${r}_M$, we give a proof of Theorem~\ref{thm:intro2} and we also show that a versal map $M\to F$ is equivalent to an injection ${r}_M\hookrightarrow{r}_F$. 

Finally, in Section~\ref{sec:functorphi} we prove Theorem~\ref{thm:introphi} by constructing the functor $\Phi$ from the category of coherent functors to the category of $A$-algebras, and in doing so we show that the Rees algebra of a module $M$ is induced by a natural map $t_M\to h^{M^\ast}$ of coherent functors. 




\par\vspace{\baselineskip}
\noindent\textbf{Acknowledgement.} I am very thankful to David Rydh for all his help and encouragement, as well as to all of our discussions. Also, I am thankful to Roy Skjelnes for his helpful input on this text. Finally, I thank Runar Ile for his many useful comments. 

\section{Versal maps}\label{sec:vers}\label{sec:1}
Throughout this paper, $A$ will denote a noetherian ring. 
Our main object of study is the Rees algebra of a finitely generated module over a noetherian ring that was defined by Eisenbud, Huneke and Ulrich in \cite{reesbud}.
\begin{definition}[{\cite[Definition~0.1]{reesbud}}]\label{def:rees}
Let $M$ be a finitely generated $A$-module. We define the \emph{Rees algebra of $M$} as the quotient ring
\[\fR(M)= \sym(M)/{\cap}_gL_g\]
where the intersection is taken over all homomorphisms $g\colon M\to E$ where $E$ runs over all free modules and $L_g=\ker\bigl(\sym(g)\colon\sym(M)\to\sym(E)\bigr)$.
\end{definition}

In \cite{jag1}, we showed the following equivalent definition.
\begin{theorem}[{\cite[Theorem~4.2]{jag1}}]\label{thm:intrinsicdef}
Let $M$ be a finitely generated $A$-module. The Rees algebra $\fR(M)$ of $M$, as defined in \cite{reesbud}, is equal to the image of the canonical map
\[\sym(M)\to\Gamma(M^\ast)^\vee\]
where  $\sym(M)$ denotes the symmetric algebra of $M$ and $\Gamma(M^\ast)^\vee$ denotes the graded dual of the algebra of divided powers $\Gamma(M^\ast)$ of the dual of the module~$M$.
\end{theorem}

\begin{remark}\label{rem:psur}
Since the symmetric algebra is graded, the Rees algebra is also graded. Furthermore, the symmetric algebra preserves surjections and it follows that the Rees algebra does as well. 
\end{remark}

To compute the Rees algebra of a module $M$ the authors of \cite{reesbud} introduced the notion of a versal map.
\begin{definition}[{\cite[Definition~1.2]{reesbud}}]
Let $M$ be a finitely generated $A$-module and let~$F$ be a finitely generated and free $A$-module. A homomorphism $\phi\colon M\to F$ is \emph{versal} if every homomorphism $M\to E$, where $E$ is free, factors via $\phi$.
\end{definition}

We now state some basic results on versal maps. Proofs can be found in \cite{reesbud}. 
\begin{proposition}\label{prop:versalgerrees}
Let $M$ be a finitely generated $A$-module, let $F$ be a finitely generated and free $A$-module, and let $\phi\colon M\to F$ be a homomorphism.
\begin{enumerate}[(i)]
\item If $\phi$ is versal, then $\fR(M)=\fR(\phi),$
where \[\fR(\phi)={\im\bigl(\sym(\phi)\colon\sym(M)\to\sym(F)\bigr)}.\] 
\item \label{item:vers2} The map $\phi$ is versal if and only if the dual $\phi^\ast\colon F^\ast\to M^\ast$ is surjective.
\item \label{item:vers3} If $\phi$ is versal then it has a canonical factorization $M\to M^{\ast\ast}\hookrightarrow F$, where $M^{\ast\ast}\hookrightarrow F$ is injective.
\end{enumerate}
\end{proposition}

\begin{proposition}\label{prop:existversal}
 Let $M$ be a finitely generated module over $A$. Then, there exists a versal map $M\to F$ for some finitely generated and free module~$F$.
\end{proposition}

\begin{remark}\label{ejversal}
Note the following.
\begin{enumerate}[(i)]
\item 
Given an ideal $I\subseteq A$, the inclusion $I\hookrightarrow A$ is \emph{not} always versal, see for instance \cite[Remark~1.6]{jag1}.
\item A homomorphism $\phi\colon M\to F$ that factors as $M\to M^{\ast\ast}\hookrightarrow F$ is not necessarily versal, see for instance \cite[Remark~1.13]{jag1}. 
\item The versal map of Proposition~\ref{prop:existversal} can be constructed as follows: choose a finitely generated and free module $F'$ that surjects onto the dual $M^\ast$. Then, the composition $M\to M^{\ast\ast}\hookrightarrow (F')^\ast$ is versal, as its dual is surjective by construction.
\end{enumerate}
\end{remark}

Two immediate consequences of the construction of the Rees algebra are the following. 
\begin{proposition}\label{prop:dsri}
Let $M$ and $N$ be finitely generated $A$-modules and let $f\colon M\to N$ be a homomorphism. 
\begin{enumerate}[(i)]
\item If $f$ is surjective, then $\fR(M)\to\fR(N)$ is surjective.
\item If $f^\ast\colon N^\ast\to M^\ast$ is surjective, then $\fR(M)\to\fR(N)$ is injective. 
\end{enumerate}
\end{proposition}
\begin{proof}
The map $\fR(M)\to\fR(N)$ is induced by the following commutative diagram.
\[\xymatrix@R=1.1em{
\sym(M)\ar[rr]\ar[dd]\ar@{->>}[dr]&&\sym(N)\ar[dd]\ar@{->>}[dr]\\
&\fR(M)\ar@{-->}[rr]\ar@{_(->}[dl]&&\fR(N)\ar@{_(->}[dl]\\
\Gamma(M^\ast)^\vee\ar[rr]&&\Gamma(N^\ast)^\vee}\]
\begin{enumerate}
\item If $f$ is surjective, then $\sym(M)\to\sym(N)$ is surjective. The diagram then implies that $\fR(M)\to\fR(N)$ is surjective.
\item If $f^\ast \colon N^\ast\to M^\ast$ is surjective, then, as the functor $\Gamma$ preserves surjections, we get a surjection on the graded map $\Gamma(N^\ast)\to\Gamma(M^\ast)$. This graded map will be a surjection in every degree, and we get, by taking the graded dual, an injection $\Gamma(M^\ast)^\vee\to\Gamma(N^\ast)^\vee$. The diagram then induces an injection ${\fR(M)\to\fR(N)}$.\qedhere
\end{enumerate}
\end{proof}
\begin{remark}
The proof of Proposition~\ref{prop:dsri} was written using the results of \cite{jag1}, but it can also be shown by using versal maps.
\end{remark}

\section{The torsionless quotient}\label{sec:torsionless}\label{sec:2}
A module $M$ is called \emph{torsionless} if it can be embedded in some free module. This is equivalent to the canonical map $M\to M^{\ast\ast}$ being injective, see, e.g., \cite[Section~4H]{MR1653294}. 
\begin{definition}
For any module $M$ we call the module $M^{tl}:=\im(M\to M^{\ast\ast})$ the \emph{torsionless quotient} of $M$. 
\end{definition}
By definition the torsionless quotient of $M$ injects into the double dual $M^{\ast\ast}$, and if the canonical map $M\to M^{\ast\ast}$ is injective, then $M^{tl}=M$. That is, if $M$ is torsionless, then $M^{tl}=M$. 
Moreover, Proposition~\ref{prop:versalgerrees}~\textit{(\ref{item:vers3})} implies that the torsionless quotient of $M$ is equal to the image $\im(M\to F)$, for any versal map $M\to F$.

\begin{lemma}\label{lem:tordual}
Given a versal map $M\to F$, then the induced map $M^{tl}\to F$ is also versal. In fact, the dual of $M$ is equal to $(M^{tl})^\ast$.
\end{lemma}
\begin{proof}
By the motivation above, the diagram 
\[\xymatrix@R=1.81em{
M\ar[rr]\ar@{->>}[dr]&& F\\
&M^{tl}\ar@{^{(}->}[ur]
}\]
is commutative. Since $M\to F$ is versal it follows by Proposition~\ref{prop:versalgerrees}~\textit{(\ref{item:vers2})} that the upper arrow in the dual diagram 
\[\xymatrix@R=1.81em{
M^\ast&& F^\ast\ar[dl]\ar@{->>}[ll]\\
&(M^{tl})^\ast\ar@{_{(}->}[ul]
}\]
is surjective. Thus, $(M^{tl})^\ast\to M^\ast$ is surjective. Since $M\to M^{tl}$ is surjective by definition, implying that $(M^{tl})^\ast\to M^\ast$ is also injective, we conclude that $(M^{tl})^\ast\to M^\ast$ is an isomorphism. As the dual $F^\ast\to M^\ast$ is surjective and $M^\ast=(M^{tl})^\ast$, we get that~${F^\ast\to(M^{tl})^\ast}$ is surjective. By Proposition~\ref{prop:versalgerrees}~\textit{(\ref{item:vers2})} it follows that $M^{tl}\to F$ is versal.
\end{proof}
This result shows that the torsionless quotient of $M$ is, as the name suggests, torsionless. Indeed, we get that the canonical map $M^{tl}\to(M^{tl})^{\ast\ast}=M^{\ast\ast}$ is injective. 
We now show that the torsionless quotient is related to the Rees algebra. 

\begin{lemma}\label{lem:1123}
Let $M$ be a finitely generated $A$-module. Then, the degree $1$ part of the graded $A$-algebra $\fR(M)$ is $M^{tl}$.
\end{lemma}
\begin{proof}
The Rees algebra can be computed as the image of the graded ${\text{$A$-algebra}}$ homomorphism $\sym(M)\to\Gamma(M^\ast)^\vee$. 
Thus, the degree $1$ part of~$\fR(M)$ is equal to the degree one part of $\im\bigl(\sym(M)\to\Gamma(M^\ast)^\vee)$, which is 
\[\im\bigl(\sym^1(M)\to\Gamma^1(M^\ast)^\vee\bigr)=\im(M\to M^{\ast\ast})=M^{tl}.\qedhere\]
\end{proof}

\begin{lemma}\label{lem:reqtl}
For every finitely generated $A$-module $M$ we have an equality ${\fR(M)=\fR(M^{tl})}$. 
\end{lemma}
\begin{proof}
Let $M\to F$ be a versal map. This factorizes as ${M\twoheadrightarrow M^{tl}\hookrightarrow F}$.
 Since also $M^{tl}\hookrightarrow F$ is versal we get a commutative diagram
\[\xymatrix{
&&\fR(M^{tl})\ar@{_(->}[d]\\
\sym(M)\ar@{->>}[r]\ar@/_1pc/@{->>}[drr]&\sym(M^{tl})\ar@{->>}[ur]\ar[r]\ar@{-->>}[dr]^\theta&\sym(F)\\
&&\fR(M)\ar@{^(->}[u]}\]
where the surjection $\theta\colon\sym(M^{tl})\twoheadrightarrow \fR(M)$ is canonically induced by the commutativity of the lower triangle. Thus, \[\fR(M^{tl})=\im\bigl(\sym(M^{tl})\to\sym(F)\bigr)=\fR(M).\qedhere\]
\end{proof}

We noted in Remark~\ref{rem:psur} that the Rees algebra preserves surjections, but an immediate consequence of the previous results is the following stronger statement.
\begin{proposition}\label{prop:tlsiffrs}
Let $M\to N$ be a homomorphism of finitely generated $A$-modules. Then, the induced map $M^{tl}\to N^{tl}$ is surjective if and only if $\fR(M)\to~\fR(N)$ is surjective. 
\end{proposition}
\begin{proof}
Suppose first that $M^{tl}\to N^{tl}$ is surjective. Then, $\fR(M^{tl})\to\fR(N^{tl})$ is surjective since the Rees algebra preserves surjections. From Lemma~\ref{lem:reqtl} we have that $\fR(M)=\fR(M^{tl})$ and $\fR(N)=\fR(N^{tl})$, so $\fR(M)\to\fR(N)$ is surjective.

Conversely, suppose that $\fR(M)\to\fR(N)$ is surjective. In particular, this map will be surjective in degree $1$, and by Lemma~\ref{lem:1123} this implies that ${M^{tl}\to N^{tl}}$ is surjective.
\end{proof}

\section{Versal maps and coherent functors}\label{sec:verfunc}\label{sec:3}
So far, we have been working in the abelian category $\moda$ of finitely generated $A$-modules over a noetherian ring $A$. Another abelian category of interest is the category $\funa$ of additive covariant functors $\fF\colon\moda\to\moda$, where kernels, cokernels and images are all calculated pointwise. In $\moda$ the notions of monomorphisms and epimorphisms are equivalent to the notions of injections and surjections. A monomorphism of functors is a morphism that is injective at every point. We therefore call a monomorphism of functors an injection. Similarly, epimorphisms of functors are pointwise surjections, and we call an epimorphism of functors a surjection.
\begin{example}
For any finitely generated module $M$, there is an additive covariant functor $h^M\colon\moda\to\moda$ defined by \scalebox{0.99}[1.0]{$N\mapsto h^M(N)=\hom_A(M,N)$}. Another example is the additive covariant functor $t_M\colon\moda\to\moda$ defined by $N\mapsto t_M(N)=M\otimes_AN$.
\end{example}
The Yoneda embedding gives a contravariant left-exact embedding of categories \[\moda\subset\funa,\] sending a finitely generated module $M$ to the additive functor $h^M$. An immediate consequence of Yoneda's lemma is that the functors $h^M$ are projective objects in $\funa$. A functor $\fF\in\funa$ is called \emph{coherent} if it has a projective resolution of the form \[h^M\to h^N\to\fF\to0,\] where $M,N\in\moda$, and $0$ denotes the zero functor. The category $\mathcal{C}$ of coherent functors  is a full subcategory of $\funa$ and has many interesting properties, 
see \cite{auslander} and \cite{coherenthart}.
\begin{ex2}
Let $M$ be a finitely generated $A$-module.
\begin{enumerate}
\item The functor $h^M$ is coherent. Indeed, it has the trivial presentation \[0=h^0\to h^M\to h^M\to0.\]
\item The functor $t_M$ is coherent. This follows since $M$ admits a projective resolution $P_1\to P_2\to M\to0$, where $P_1$ and $P_2$ are finitely generated projective modules. Since tensoring is right-exact we get an exact sequence $t_{P_1}\to t_{P_2}\to t_M\to 0$. For finitely generated projective modules $P$ it holds that $t_P=h^{P^\ast}$, giving a projective resolution $h^{P_1^\ast}\to h^{P_2^\ast}\to t_M\to0$. \exendhere
\end{enumerate}
\end{ex2}

This example shows that the Yoneda embedding $\moda\subset\funa$ actually takes values in~$\mathcal{C}$. That is, the functor $M\mapsto h^M$ is an embedding of $\moda$ into the category $\mathcal{C}$ of coherent functors. 
\begin{theorem}[{\cite[Theorem~1.1a]{coherenthart}}]\label{thm:puuh}
If $f\colon\fF_1\to\fF_2$ is a morphism of coherent functors, then $\ker(f)$, $\coker(f)$, and $\im(f)$ are also coherent.
\end{theorem}
One advantage that the category of coherent functors has over the category of finitely generated modules is that it has an {exact} and reflexive dual.

\begin{proposition}[{\cite[Proposition~4.1]{coherenthart}}]\label{prop:dualityoffun}
Let $\mathcal{C}$ denote the category of coherent functors. There is a unique functor $\vee\colon\mathcal{C}\to\mathcal{C}$ which is exact, contravariant, and has the property that $\vee(h^M)=t_M$ for every finitely generated module $M$. Furthermore, $\vee\vee\cong\id_\mathcal{C}$.
\end{proposition}
\begin{remark}
In the sequel we write $\fF^\vee:=\vee(\fF)$ for any coherent functor $\fF$. Note also that we use the same symbol for the dual of a coherent functor as we do for the graded dual of a graded $A$-algebra. This is to better emphasize the analogies that we will see further on.
\end{remark} 
Analogously to the Yoneda embedding we can consider the functor $t_{-}\colon\moda\to\mathcal{C}$, defined by $M\mapsto t_M$, which gives a covariant right-exact embedding $\moda\subset\mathcal{C}$. There is also a functor $\operatorname{ev}_A\colon\funa\to\moda$ defined by evaluating functors at $A$, i.e., $\fF\mapsto\fF(A)$. Note that ${M\mapsto t_M\mapsto t_M(A)=M}$ is the identity, so evaluating the functor at $A$ is a section of the embedding $M\mapsto t_M$, and in the spirit of algebraic geometry we call this taking global sections.

Let now $\phi\colon M\to F$ be a versal map. This gives a morphism $t_{M}\to t_{F}$ of coherent functors.
We saw in Proposition~\ref{prop:versalgerrees}~\textit{(\ref{item:vers3})} that such a versal map factors as $M\to M^{\ast\ast}\hookrightarrow F$, and embedding this composition into the category of coherent functors gives a commutative diagram
\[\xymatrix@R=1.7em{t_{M^{\ast\ast}}\ar[dr]\\
t_M\ar[r]\ar[u]&t_F}\]
where $t_{M^{\ast\ast}}\to t_F$ is \emph{not} injective in general, since tensoring is only right-exact. It turns out that there is another functor, not $t_{M^{\ast\ast}}$, through which the map $t_M\to t_F$ has a canonical factorization, such that the second map is injective. We will show that this functor is~$h^{M^{\ast}}$, resulting in the factorization \eqref{eq:hinj}.

\begin{proposition}[{\cite[Proposition~3.1]{coherenthart}}]\label{prop:naturalmap}
Let $\fF$ be a (not necessarily coherent) functor. Then there is a natural map $\alpha\colon \fF(A)\otimes_A(-)\to \fF$. 
Furthermore, $\fF$ is right-exact if and only if $\alpha$ is an isomorphism.
\end{proposition}
\begin{remark}
Given an $A$-module $N$, the $A$-module homomorphism \[\alpha_N\colon\fF(A)\otimes_AN\to\fF(N)\] is defined by sending, for all $a\in\fF(A)$ and all $n\in N$, the element $a\otimes n$ to the element $\fF(s_{n})(a)$, where $s_{n}\colon A\to N$ is defined by $1\mapsto n$. 

We have that 
 ${\fF(A)\otimes_A(-)=t_{\operatorname{ev}_A(\fF)}(-)}$, so the proposition implies that there is a natural transformation $t_{\operatorname{ev}_A(-)}\to\id_{\mathcal{C}}$. For every module $M$ we have that $\operatorname{ev}_A(t_M)=t_M(A)=M$, showing that there is also a natural transformation $\id_{\moda}\to\operatorname{ev}_A(t_{(-)})$. One can in fact show that these natural transformations are units/counits of an adjunction between $t_{(-)}$ and~$\operatorname{ev}_A$. 
\end{remark}

Applying Proposition~\ref{prop:naturalmap} to the functor $h^M$ gives a morphism ${h^M(A)\otimes_A(-)\to h^M}$. Noting that ${h^M(A)\otimes_A(-)=t_{M^\ast}}$, we rewrite this morphism as $t_{M^\ast}\to h^M$. 
Moreover, given a versal map $M\to F$ we have that $F^\ast\to M^\ast$ is surjective. Tensoring is right-exact, which induces a surjection $t_{F^\ast}\twoheadrightarrow t_{M^\ast}$, and since $F$ is free of finite rank we have an isomorphism $t_{F^{\ast}}=h^F$. Thus, Proposition~\ref{prop:naturalmap} shows that a versal map $M\to F$ induces a commutative diagram
\begin{equation}\label{eq:faktofh}
\begin{split}
\xymatrix{t_{M^\ast}\ar[d]&t_{F^\ast}\ar@{=}[d]\ar@{->>}[l]\\
h^M&h^F\ar@{->>}[ul]\ar[l]}
\end{split}
\end{equation}
which we can dualize to get the commutative diagram:
\begin{equation}\label{eq:hinj}
\begin{split}
\xymatrix{h^{M^\ast}\ar@{^(->}[dr]&\\
t_M\ar[u]\ar[r]&t_F}
\end{split}
\end{equation}
Taking global sections gives that $h^{M^\ast}(A)=M^{\ast\ast}$, recovering the original factorization of Proposition~\ref{prop:versalgerrees}~\textit{(\ref{item:vers3})}. However, here we see that the fact that versal maps factor via the double dual is just a special case of taking global sections of a functor $h^{M^\ast}=(t_{M^\ast})^\vee$ with the \emph{two different duals} $\ast$ and $\vee$. With this realization we are able to state a generalization of Proposition~\ref{prop:versalgerrees}~\textit{(\ref{item:vers2})}. 

\begin{theorem}\label{thm:mainversal}
Let $A$ be a noetherian ring, let $M$ be a finitely generated $A$-module and let $F$ be a finitely generated and free $A$-module. Given a homomorphism $\phi\colon M\to F$, the following are equivalent:
\begin{enumerate}[(i)]
\item\label{item:1} $\phi$ is versal.
\item\label{item:2} $\phi^\ast\colon F^\ast\to M^\ast$ is surjective.
\item\label{item:3} $\im(h^F\to h^M)=\im(t_{M^\ast}\to h^M)$.
\item\label{item:4} $\im(t_M\to t_F)=\im(t_M\to h^{M^\ast})$.
\item\label{item:5} $h^{M^\ast}\to t_F$ is injective.
\end{enumerate}
\end{theorem}
\begin{proof} We prove this by showing the following equivalences: \emph{(\ref{item:1})}$\Leftrightarrow$\emph{(\ref{item:2})}, \emph{(\ref{item:2})}$\Leftrightarrow$\emph{(\ref{item:3})}, \emph{(\ref{item:3})}$\Leftrightarrow$\emph{(\ref{item:4})}, and \emph{(\ref{item:2})}$\Leftrightarrow$\emph{(\ref{item:5})}.

\noindent\emph{(\ref{item:1})}$\Leftrightarrow$\emph{(\ref{item:2})}:  
This is Proposition~\ref{prop:versalgerrees}~\textit{(\ref{item:vers2})}. 

\noindent\emph{(\ref{item:2})}$\Leftrightarrow$\emph{(\ref{item:3})}: Suppose $F^\ast\to M^\ast$ is surjective.
This gives the factorization \eqref{eq:faktofh},
from which it follows that $\im(h^F\to h^M)=\im(t_{M^\ast}\to h^M)$.

Conversely, suppose that $\im(h^F\to h^M)=\im(t_{M^\ast}\to h^M)$. Taking global sections, we get
\begin{equation*}\im(F^\ast\to M^\ast)=\im(M^\ast\to M^\ast)=M^\ast.\end{equation*}
That is, $F^\ast\to M^\ast$ is surjective. 

\noindent\emph{(\ref{item:3})}$\Leftrightarrow$\emph{(\ref{item:4})}: These are dual statements of each other by Proposition~\ref{prop:dualityoffun}.

\noindent\emph{(\ref{item:2})}$\Leftrightarrow$\emph{(\ref{item:5})}: That $F^\ast\to M^\ast$ is surjective is equivalent to $t_{F^\ast}\to t_{M^\ast}$ being surjective. Dualizing gives that the map $h^{M^\ast}\to t_F$ is injective.
\end{proof}

\begin{remark}
By Theorem~\ref{thm:mainversal},  a map $M\to F$ is versal if and only if the induced map $h^{M^\ast}\to t_F$ is injective. This is a result that does \emph{not} have an analogous statement in the category of modules. Indeed, taking global sections of \emph{(\ref{item:5})} gives an injection $M^{\ast\ast}\hookrightarrow F$, but, as we saw in Remark~\ref{ejversal}, an $A$-module homomorphism $M\to F$ need not be versal even though it factors as $M\to M^{\ast\ast}\hookrightarrow F$. Instead, this theorem shows that a versal map $M\to F$ is precisely a map that induces a factorization $t_M\to h^{M^\ast}\hookrightarrow t_F$.
\end{remark}

\section{Torsionless functors}\label{sec:tf}\label{sec:4}
Analogously to the definition of the torsionless quotient of a module from Section~\ref{sec:torsionless}, we will now consider the torsionless quotient in the category of coherent functors.
\begin{definition}
Given a finitely generated module $M$ we define the \emph{torsionless quotient functor of $M$} as the image of the canonical map $t_M\to h^{M^\ast}$ and denote it by ${r}_M$, that is, 
\[{r}_M=\im(t_M\to h^{M^\ast}).\]
\end{definition}
\begin{remark}\label{rem:2a2}
By Theorem~\ref{thm:puuh} it is clear that ${r}_M$ is a coherent functor for any finitely generated module $M$. Furthermore, we note that ${{r}_M(A)=\im(M\to M^{\ast\ast})=M^{tl}}$. However, in general, ${r}_M$ is neither given by $t_{M^{tl}}$ nor by the image of the map $t_M\to t_{M^{\ast\ast}}$.
By dualizing we have that ${r}_M^\vee=\im(t_{M^\ast}\to h^M)$, and, in particular, 
\[{r}_M^\vee(A)=\im(M^\ast\to M^\ast)=M^\ast.\]
Also, as the functor $t_M$ preserves surjections it follows that ${r}_M$ does as well.
\end{remark}

\begin{lemma}\label{lem:likamedtl}
If $M$ is a finitely generated module, then ${r}_M={r}_{M^{tl}}$.
\end{lemma}
\begin{proof}
By Lemma~\ref{lem:tordual} we have a sequence $t_M\twoheadrightarrow t_{M^{tl}}\to h^{(M^{tl})^\ast}=h^{M^\ast}$ giving the result.
\end{proof}

\begin{lemma}
If $F$ is a free and finitely generated module, then ${r}_F=t_F$.
\end{lemma}
\begin{proof}
If $F$ is free and finitely generated then 
$t_F=h^{F^\ast}$, and the result follows.
\end{proof}

\begin{lemma}\label{lem:gfunctor}
The map $M\mapsto{r}_M$ naturally extends to a covariant functor ${{r}_{-}\colon \moda\to \mathcal{C}}$.
\end{lemma}
\begin{proof}
Let $f\colon M\to N$ be a module homomorphism. This gives a commutative diagram
\[\xymatrix@R=1.3em {
t_M\ar[rr]\ar@{->>}[rd]\ar[dd]&&t_N\ar@{->>}[dr]\ar[dd]\\
&{r}_M\ar@{_(->}[dl]&&{r}_N\ar@{_(->}[dl]\\
h^{M^\ast}\ar[rr]&&h^{N^{\ast}}}\]
which induces a morphism ${r}_f\colon{r}_M\to{r}_N$. The other defining properties of a functor follow by similar arguments. 
\end{proof}

\begin{remark}
A morphism ${r}_M\to{r}_N$ induces a homomorphism ${M^{tl}\to N^{tl}}$ by taking global sections. Moreover, a morphism ${r}_M\to{r}_N$ is uniquely determined by the map $M^{tl}\to N^{tl}$ as the following result shows.
\end{remark}


\begin{proposition}\label{prop:tlmotmf}
Let $M$ and $N$ be finitely generated modules over $A$. Then, a morphism $u\colon{r}_M\to{r}_N$ is uniquely determined by the map $M^{tl}\to N^{tl}$ obtained by evaluating at $A$.
\end{proposition}
\begin{proof}
Evaluating at $A$ gives a map $f\colon M^{tl}\to N^{tl}$. We need to show that $u={r}_f$. By the functoriality of the map from Proposition~\ref{prop:naturalmap}, we have a commutative diagram
\[\xymatrix{{r}_M(A)\otimes_A P\ar[r]^-{u_A\otimes\id}\ar[d]&{r}_N(A)\otimes_A P\ar[d]\\
{r}_M(P)\ar[r]^{u_P}&{r}_N(P)}\]
for any finitely generated module $P$. As ${r}_M(A)=M^{tl}$, ${r}_N(A)=N^{tl}$, and $u_A=f$, this reduces to the commutative diagram
\[\xymatrix{t_{M^{tl}}(P)\ar[r]^-{f\otimes\id}\ar@{->>}[d]&t_{N^{tl}}(P)\ar@{->>}[d]\\
{r}_M(P)\ar[r]^{u_P}&{r}_N(P)}\]
where, again by the functoriality of the map from Proposition~\ref{prop:naturalmap}, the vertical arrows are surjective. Thus, ${r}_f(P)=u_P$ for any module $P$.
\end{proof}

By this result we have that morphisms ${r}_M\to{r}_N$, of torsionless quotient functors, are equivalent to homomorphisms $M^{tl}\to N^{tl}$, of torsionless modules. From Lemma~\ref{lem:likamedtl}, it now follows that the full subcategory of torsionless quotient functors of $\mathcal{C}$ is equivalent to the full subcategory of torsionless modules of the category of finitely generated modules $\moda$. 

Even though these categories are equivalent, we end this section by showing that there is an advantage in working with the category of torsionless functors, rather than the category of torsionless modules. In particular, Corollary~\ref{cor:gandr} gives connections between torsionless functors and Rees algebras that do not exist between torsionless modules and Rees algebras.

\begin{proposition}\label{prop:is}
Let $f\colon M\to N$ be a homomorphism of finitely generated $A$-modules.
\begin{enumerate}[(i)]
\item\label{item:inj} The map ${r}_f\colon {r}_M\to {r}_N$ is injective if and only if $f^\ast\colon N^\ast\to M^\ast$ is surjective.
\item The map ${r}_f\colon {r}_M\to{r}_N$ is surjective if and only if the induced map ${M^{tl}\to N^{tl}}$ is surjective.
\end{enumerate} 
\end{proposition}
\begin{proof}
As the map $t_M\to t_{M^{tl}}$ is surjective we have that the diagram
\begin{equation}\label{eq:injsur}
\begin{split}
\xymatrix{t_{M^{tl}}\ar@{->>}[r]\ar[d]&{r}_M\ar@{^(->}[r]\ar[d]&h^{M^\ast}\ar[d]\\ 
t_{N^{tl}}\ar@{->>}[r]&{r}_N\ar@{^(->}[r]&h^{N^\ast}}
\end{split}
\end{equation}
commutes.
\begin{enumerate}[(i)]
\item If ${r}_M\to {r}_N$ is injective then, dually, ${r}_N^\vee\to{r}_M^\vee$ is surjective. Taking global sections, we get that $N^\ast\to M^\ast$ is surjective. Conversely, if $N^\ast\to M^\ast$ is surjective, then $h^{M^\ast}\to h^{N^\ast}$ is injective, so the commutativity of diagram \eqref{eq:injsur} implies that ${r}_M\to{r}_N$ is injective.
\item Suppose that ${r}_M\to{r}_N$ is surjective. Then, taking global sections gives a surjection $M^{tl}\to N^{tl}$. Conversely, if $M^{tl}\twoheadrightarrow N^{tl}$ is surjective, 
then $t_{M^{tl}}\to t_{N^{tl}}$ is surjective, so the diagram~\eqref{eq:injsur} implies that ${r}_M\to{r}_N$ is surjective.
\qedhere\end{enumerate}
\end{proof}

An immediate consequence of combining the previous result with Propostion~\ref{prop:versalgerrees}~\textit{(\ref{item:vers2})} is the following.
\begin{corollary}
Let $M$ be a finitely generated $A$-module and let $F$ be a finitely generated and free $A$-module. Then, a homomorphism $M\to F$ is versal if and only if ${r}_M\to{r}_F$ is injective.
\end{corollary}
Similarly, combining Proposition~\ref{prop:is} with Proposition~\ref{prop:dsri} and Proposition~\ref{prop:tlsiffrs} gives:
\begin{corollary}\label{cor:gandr}
Let $M\to N$ be a homomorphism of finitely generated $A$-modules.
\begin{enumerate}[(i)]
\item If the induced map ${r}_M\to{r}_N$ is injective, then the induced algebra homomorphism ${\fR(M)\to\fR(N)}$ is injective.
\item The map ${r}_M\to{r}_N$ is surjective if and only if $\fR(M)\to\fR(N)$ is surjective.
\end{enumerate}
\end{corollary}

\section{A functor from coherent functors to commutative algebras}\label{sec:functorphi}

In this section we construct a functor $\Phi\colon\mathcal{C}\to\mathbf{Alg}_A$ from the category of coherent functors to the category of finitely generated and graded $A$-algebras. This we do by first defining $\Phi(h^M)$ for every finitely generated module $M$. Then, for every coherent functor $\fF$, we fix a presentation $h^N\overset{u}{\to} h^M\to\fF\to0$, and let $\Phi(\fF)$ be the coequalizer $\coeq\bigl(\Phi(u),\Phi(0)\bigr)$. Finally, we show that this construction is independent of the choice of presentation of $\fF$. 

By Theorem~\ref{thm:intrinsicdef}, the Rees algebra of $M$ is equal to the image of the canonical map $\sym(M)\to\Gamma(M^\ast)^\vee$. 
For definitions and properties of the algebra of divided powers and its dual, we refer the reader to \cite{Roby1963}, \cite{Rydh} or \cite{jag1}. From the latter, we have the following result.
\begin{theorem}[{\cite[Theorem~3.9]{jag1}}]\label{thm:gammadualissym}
Let $M$ be a finitely generated module over $A$. Then, the canonical module homomorphism $M^\ast\to \Gamma(M)^\vee$, sending $M^\ast$ into the degree 1 part of $\Gamma(M)^\vee$,  which is $M^\ast$, induces a natural homomorphism of graded $A$-algebras
\[\sym(M^\ast)\to\Gamma(M)^\vee.\] If $M$ is free, then this map is an isomorphism.
\end{theorem}
By Remark~3.10 of \cite{jag1}, we also note that $\Gamma(M)^\vee$ is generally \emph{not} generated in degree~$1$, and is therefore hard to work with. Instead, we will consider the $A$-algebra
\[\mathcal{Q}^{\operatorname{op}}(M):=\im\bigl(\sym(M^\ast)\to\Gamma(M)^\vee\bigr),\]
which is the largest subring of $\Gamma(M)^\vee$ that is generated in degree $1$. Furthermore, for any finitely generated $A$-module $M$, we define 
\[\mathcal{Q}(M):=\mathcal{Q}^{\operatorname{op}}(M^\ast)=\im\bigl(\sym(M^{\ast\ast})\to\Gamma(M^\ast)^\vee\bigr).\] 
\begin{lemma}\label{lem:stillrees}
The Rees algebra $\fR(M)$ of a finitely generated $A$-module $M$ is equal to the image of the canonical map $\sym(M)\to\mathcal{Q}(M)$. When $M$ is reflexive there is even an equality ${\fR(M)=\mathcal{Q}(M)}$.
\end{lemma}
\begin{proof}
From Theorem~\ref{thm:intrinsicdef}, we have that $\fR(M)=\im\bigl(\sym(M)\to\Gamma(M^\ast)^\vee\bigr)$. By the universal property of the symmetric algebra there is a factorization \[\sym(M)\to\sym(M^{\ast\ast})\to\Gamma(M^\ast)^\vee,\] 
from which it follows that the image of $\sym(M)\to\Gamma(M^\ast)^\vee$ lies within the image of $\sym(M^{\ast\ast})\to\Gamma(M^\ast)^\vee$. 
Thus
\[\fR(M)=\im\bigl(\sym(M)\to\Gamma(M^\ast)^\vee\bigr)=\im\bigl(\sym(M)\to\mathcal{Q}(M)\bigr).\]
If $M$ is reflexive, then
\[\mathcal{Q}(M)=\im\bigl(\sym(M^{\ast\ast})\to\Gamma(M^\ast)^\vee\bigr)= \im\bigl(\sym(M)\to\Gamma(M^{\ast})^\vee\bigr)=\fR(M).\qedhere\]
\end{proof}

\begin{theorem}\label{thm:thereisafunctor}
There is a functor $\Phi\colon\mathcal{C}\to\mathbf{Alg}_A$ such that $h^M\mapsto \mathcal{Q}^{\operatorname{op}}(M)$.
\end{theorem}
To give some structure to the proof, we break it down to a few lemmas. Given any finitely generated module $M$, we let 
$\Phi(h^M)= \mathcal{Q}^{\operatorname{op}}(M)$. For every coherent functor $\fF$, we fix a projective resolution
\[h^N\overset{u}{\to} h^M\to\fF\to0,\]
so that $\fF=\coker(u)=\coeq(u,0)$. Then, we define $\Phi(\fF)=\coeq\bigl(\Phi(u),\Phi(0)\bigr)$. That this makes $\Phi$ into a well defined functor is proved by the following results.

\begin{lemma}\label{lem:detsomfixar}
Let $u\colon h^N\to h^M$ be a map of coherent functors, and let $p_1\colon h^{M\oplus N}\to h^M$ and $p_2\colon h^{M\oplus N}\to h^N$ denote the projections. Then, $\coeq\bigl(\Phi(u),\Phi(0)\bigr)=\coeq\bigl(\Phi(\pi_1),\Phi(\pi_2)\bigr)$, where $\pi_1,\pi_2\colon h^{M\oplus N}\to h^M$ are defined by $\pi_1=p_1+u\circ p_2$ and $\pi_2=p_1$.
\end{lemma}
\begin{proof}
The coequalizer of the two maps $\Phi(u),\Phi(0)\colon\mathcal{Q}^{\operatorname{op}}(N)\to\mathcal{Q}^{\operatorname{op}}(M)$ is equal to the graded $A$-algebra $\mathcal{Q}^{\operatorname{op}}(M)/I$ where $I$ is the ideal generated by all elements of the form \[\Phi(u)(x)-\Phi(0)(x)\] for all $x\in\mathcal{Q}^{\operatorname{op}}(N)$. As $\mathcal{Q}^{\operatorname{op}}(N)$ is generated in degree $1$, it follows that $I$ is generated by $\Phi(u)(x)$ for all $x$ in the degree~$1$ part of $\mathcal{Q}^{\operatorname{op}}(N)$, which is $N^\ast$. Similarly, the coequalizer of $\Phi(\pi_1)$ and $\Phi(\pi_2)$ is $\mathcal{Q}^{\operatorname{op}}(M)/J$, where $J$ is the ideal generated by elements \[\Phi(\pi_1)(x,y)-\Phi(\pi_2)(x,y)\] for all $(x,y)$ in $M^\ast\oplus N^\ast$. For every ${(x,y)\in M^\ast\oplus N^\ast}$ we have that 
\[\Phi(\pi_1)(x,y)-\Phi(\pi_2)(x,y)=\Phi(u)(y),\] so $I=J$.
\end{proof}

\begin{lemma}\label{lem:phiavmorfier}
Consider a map $f\colon\fF\to\mathcal{G}$ of coherent functors. Then, there is a  natural map $\Phi(f)\colon\Phi(\fF)\to\Phi(\mathcal{G})$.
\end{lemma}
\begin{proof}
Let $h^N\to h^M\to\fF\to0$ and $h^L\to h^P\to\mathcal{G}\to0$ be fixed presentations of $\fF$ and~$\mathcal{G}$. Then, a map $\fF\to\mathcal{G}$ lifts to a map of complexes:
\[\xymatrix{
h^N\ar[r]^{u_1}\ar[d]^{f_2}&h^M\ar[r]\ar[d]^{f_1}&\fF\ar[r]\ar[d]^f&0\\
h^L\ar[r]^{u_2}&h^P\ar[r]^-p&\mathcal{G}\ar[r]&0}\]
Applying $\Phi$ to this diagram induces a map of coequalizers 
\[\Phi(\fF)=\coeq\bigl(\Phi(u_1),\Phi(0)\bigr)\to\coeq\bigl(\Phi(u_2),\Phi(0)\bigr)=\Phi(\mathcal{G}).\]
We need to show that this map is independent of the choice of $f_1,f_2$ as a lift of $f$. Given another lift $g_1,g_2$ of $f$,  we get a homotopy $l\colon h^M\to h^L$ such that $u_2\circ l=f_1-g_1$. Now, letting $p_1\colon h^{P\oplus L}\to h^P$ and $p_2\colon h^{P\oplus L}\to h^L$ denote the projections, we define $\pi_1=p_1+u_2\circ p_2$ and $\pi_2=p_1$ as in Lemma~\ref{lem:detsomfixar}. Then, we consider the new diagram:
\[\xymatrix{
h^{N}\ar[r]^{u_1} \ar@<0.5ex>[d]\ar@<-0.5ex>[d] &h^M\ar[r]\ar@<0.5ex>[d]^{f_1}\ar@<-0.5ex>[d]_{g_1}&\fF\ar[r]\ar[d]^f&0\\
h^{P\oplus L}\ar@<0.5ex>[r]^{\pi_1}\ar@<-0.5ex>[r]_{\pi_2}&h^P\ar[r]^{p}&\mathcal{G}\ar[r]&0}\]
Letting $j_1\colon h^P\to h^{P\oplus L}$ and $j_2\colon h^L\to h^{P\oplus L}$ denote the natural inclusions, we have that 
\[\pi_1\circ(j_1\circ g_1+j_2\circ l)=(p_1+u_2\circ p_2)\circ(j_1\circ g_1+j_2\circ l) = g_1+u_2\circ l=f_1\] and 
\[\pi_2\circ(j_1\circ g_1+j_2\circ l)=p_1\circ(j_1\circ g_1+j_2\circ l)=g_1.\] 
Applying $\Phi$ now gives $\Phi(f_1)=\Phi(\pi_1)\circ\Phi(j_1\circ g_1+j_2\circ l)$
and $\Phi(g_1)=\Phi(\pi_2)\circ\Phi(j_1\circ g_1+j_2\circ l)$.
Hence, for every $x\in\mathcal{Q}^{\operatorname{op}}(M)$, we see that the element \[y=\Phi(j_1\circ g_1+j_2\circ l)(x)\in\mathcal{Q}^{\operatorname{op}}(P\oplus L)\] has the property that $\Phi(\pi_1)(y)=\Phi(f_1)(x)$ and $\Phi(\pi_2)(y)=\Phi(g_1)(x)$. This shows that $\Phi(p)\circ\Phi(f_1)=\Phi(p)\circ\Phi(g_1)$. Using Lemma~\ref{lem:detsomfixar} we now get the diagram 
\[\xymatrix{
\mathcal{Q}^{\operatorname{op}}(N)\ar@<0.5ex>[r]^-{\Phi(u_1)}\ar@<-0.5ex>[r]_-{\Phi(0)} \ar@<0.5ex>[d]\ar@<-0.5ex>[d] &\mathcal{Q}^{\operatorname{op}}(M)\ar[r]\ar@<0.5ex>[d]^{\Phi(f_1)}\ar@<-0.5ex>[d]_{\Phi(g_1)}&\Phi(\fF)\ar[r]\ar@{-->}[d]&0\\
\mathcal{Q}^{\operatorname{op}}({P\oplus L})\ar@<0.5ex>[r]^-{\Phi(\pi_1)}\ar@<-0.5ex>[r]_-{\Phi(\pi_2)}&\mathcal{Q}^{\operatorname{op}}(P)\ar[r]^-{\Phi(p)}&\Phi(\mathcal{G})\ar[r]&0}\]
in which $\Phi(f_1)$ and $\Phi(g_1)$ induce the same map between $\Phi(\fF)$ and $\Phi(\mathcal{G})$. 
\end{proof}

\begin{proof}[Proof of Theorem~\ref{thm:thereisafunctor}]
As before, we set 
$\Phi(h^M)= \mathcal{Q}^{\operatorname{op}}(M)$ for every finitely generated module $M$, and we fix a projective resolution
\[h^N\overset{u}{\to} h^M\to\fF\to0,\]
for every coherent functor $\fF$. Then, we define $\Phi(\fF)=\coeq\bigl(\Phi(u),\Phi(0)\bigr)$. It remains to prove that this is independent of the choice of projective resolution. 

To show this, we take another projective resolution $h^L\to h^P\to\fF\to0$ of $\fF$. Then, we can lift the identity map $\fF\to\fF$ to a map of complexes:
\[\xymatrix@R=1.8em{
h^N\ar[r]^u\ar[d]&h^M\ar[r]\ar[d]^f&\fF\ar@{=}[d]\ar[r]&0\\
h^L\ar[r]^{u'}\ar[d]&h^P\ar[r]\ar[d]^g&\fF\ar@{=}[d]\ar[r]&0\\
h^N\ar[r]^u&h^M\ar[r]&\fF\ar[r]&0}\]
Thus, we have two lifts $\id,g\circ f\colon h^M\to h^M$. By the proof of Lemma~\ref{lem:phiavmorfier} these two lifts induce the same map $\Phi(\fF)\to\coeq\bigl(\Phi(u'),\Phi(0)\bigr)\to\Phi(\fF)$, 
and this map is the identity, so ${\Phi(\fF)=\coeq\bigl(\Phi(u'),\Phi(0)\bigr)}$. Hence, the definition of $\Phi$ is independent of the choice of projective resolution. That $\Phi$ is well defined on morphisms now follows from~Lemma~\ref{lem:phiavmorfier}.
\end{proof}

\begin{remark}
The functor $\Phi$ is not right-exact. It does preserve coequalizers, but not, in general, coproducts. Indeed, in Example~$1.15$ of \cite{jag1}, there is a \mbox{module}~$M$ such that ${\fR(M\oplus M)\neq\fR(M)\otimes_A\fR(M)}$. As that module is reflexive, we have by Lemma~\ref{lem:stillrees} that $\fR(M)=\mathcal{Q}(M)=\mathcal{Q}^{\operatorname{op}}(M^\ast)$, showing that $\mathcal{Q}^{\operatorname{op}}$ does not preserve coproducts.
\end{remark}

\begin{remark}
There is a reason for defining $\Phi(h^M)$ as $\mathcal{Q}^{\operatorname{op}}(M)$, and not as $\Gamma(M)^\vee$. 
Indeed, if we define $\Psi(h^M)=\Gamma(M)^\vee$ and consider a presentation $h^Q\overset{u}{\to} h^P\to h^M\to0$, then it is \emph{not} true that $\Gamma(M)^\vee=\coeq\bigl(\Psi(u),\Psi(0)\bigr)$. That is because Lemma~\ref{lem:detsomfixar} fails for $\mathcal{Q}^{\operatorname{op}}(M)$ replaced with $\Gamma(M)^\vee$, as the latter is not generated in degree~1.
\end{remark}

\begin{proposition}
Let $M$ be a finitely generated module over $A$. Then, ${\Phi(t_M)=\sym(M)}$.
\end{proposition}
\begin{proof}
Choosing a projective resolution $F_2\to F_1\to M\to0$ of $M$ gives a right-exact sequence $t_{F_2}\to t_{F_1}\to t_M\to0$. For projective modules it holds that $t_{F_i}=h^{F_i^\ast},$ so we get an exact sequence
\begin{equation}\label{eq:res}h^{F_2^\ast}\overset{u}{\to} h^{F_1^\ast}\to t_M\to0.\end{equation}
Since $F_1$ and $F_2$ are free, we have by Theorem~\ref{thm:gammadualissym} that $\sym(F_i)=\Gamma(F_i^\ast)^\vee$, 
and in particular ${\Phi\bigl(h^{F_i^\ast}\bigr)=\mathcal{Q}(F_i)= \sym(F_i)}$. Thus, applying $\Phi$ to \eqref{eq:res} gives the sequence
\[\xymatrix{
\sym(F_2)\ar@<0.5ex>[r]^{\Phi(u)}\ar@<-0.5ex>[r]_{\Phi(0)}&\sym(F_1)\ar[r]&\Phi(t_M)=\coeq\bigl(\Phi(u),\Phi(0)\bigr).}\]
As the symmetric algebra preserves coequalizers it follows that $\Phi(t_M)=\sym(M)$.
\end{proof}

Let us now consider a versal map $M\to F$. By combining the results stated above, we get the following commutative diagram:
\begin{equation}
\begin{split}
\xymatrix{
M\ar[r]\ar@{|->}[d]^{t_{-}}& M^{\ast\ast}\ar@{^(->}[r]& F\ar@{|->}[d]^{t_{-}}\\
t_M\ar[r]\ar@{|->}[d]^\Phi& h^{M^\ast}\ar@{^(->}[r]\ar@{|->}[d]^\Phi&t_F\ar@{|->}[d]^\Phi\\
\sym(M)\ar[r]&\mathcal{Q}(M)\ar@{^(->}[r]&\sym(F)}
\end{split}
\end{equation}
Using Lemma~\ref{lem:stillrees}, we conclude that the Rees algebra of $M$ is given by \[\fR(M)=\im\Bigl(\Phi\bigl(t_M\bigr)\to\Phi\bigl(h^{M^\ast}\bigr)\Bigr).\] 
That is, the Rees algebra of $M$ is induced from the canonical map $t_M\to h^{M^\ast}$ of coherent functors.
There is no functor from the category of modules with this property. Thus, some intrinsic properties of $M$ are better reflected in the category of coherent functors than in the category of $A$-modules.


Throughout this paper, we have seen many connections between the Rees algebra of a module $M$ and the torsionless quotient functor ${r}_M$. First of all, they are both given as images of canonical maps, 
\[{r}_M=\im\bigl(t_M\to h^{M^\ast}\bigr)\quad \text{and}\quad \fR(M)=\im\bigl(\sym(M)\to\Gamma(M^\ast)^\vee\bigr).\]
The results of this section show that also the second of these maps is induced from the canonical map $t_M\to h^{M^\ast}$.
Moreover, we showed in Sections~\ref{sec:1} and \ref{sec:2} that a versal map $M\to F$ factorizes as the composition 
\[M\twoheadrightarrow M^{tl}\hookrightarrow M^{\ast\ast}\hookrightarrow F.\]
This versal map induces a morphism $t_M\to t_F$ of coherent functors, and we showed in Sections~\ref{sec:3} and \ref{sec:4} that the previous factorization is a special case of taking global sections of the factorization given by the composition 
\[t_M\twoheadrightarrow{r}_M\hookrightarrow h^{M^\ast}\hookrightarrow t_F.\]
Analogously, we have shown that the induced map $\sym(M)\to\sym(F)$ factorizes as the composition
 \[\sym(M)\twoheadrightarrow\fR(M)\hookrightarrow\Gamma(M^\ast)^\vee\hookrightarrow\sym(F).\]
Note that $\Gamma(M^\ast)^\vee$ and $h^{M^\ast}=(t_{M^\ast})^\vee$ are both described by \emph{two different duals}. Due to the similarities of these last two factorizations, one could hope that $\Phi$ would map ${r}_M$ to $\fR(M)$, but that is not the case as the following result~shows. 

\begin{proposition}
Let $M$ be a finitely generated module over $A$. Then, $\Phi({r}_M)=\sym(M^{tl})$.
\end{proposition}
\begin{proof}
Let ${r}_M=\im\big(t_M\to h^{M^\ast}\big)$. Choose a versal map $M\to F$ and a surjection $E\to M$ from a free module $E$. Then, ${r}_M=\im(t_E\to t_F)$. Let $L=\coker(F^\ast\to E^\ast)$ so that $h^{L}=\ker(t_E\to t_F)$. 
In particular, we have that \[{r}_M=\im(t_E\to t_F)=\coker(f\colon h^L\to t_E).\] Thus, $\Phi({r}_M)=\coeq\bigl(\Phi(f),\Phi(0)\bigr)=\coeq\bigl(\mathcal{Q}^{\operatorname{op}}(L)\rightrightarrows\sym(E)\bigr)$. By definition, there is a surjection $\sym(L^\ast)\to \mathcal{Q}^{\operatorname{op}}(L)$, so it follows that 
\[\Phi({r}_M)=\coeq\bigl(\mathcal{Q}^{\operatorname{op}}(L)\rightrightarrows\sym(E)\bigr)=\coeq\bigl(\sym(L^\ast)\rightrightarrows\sym(E)\bigr).\] 
As the symmetric algebra preserves colimits, we get that
\begin{align*}\Phi({r}_M)=\coeq\bigl(\sym(L^\ast)\rightrightarrows\sym(E)\bigr)=\sym\bigl(\coeq(L^\ast\rightrightarrows E)\bigr)=\sym\bigl(\im(E\to F)\bigr).\end{align*}
Since $M^{tl}=\im(M\to F)=\im(E\to F)$ we conclude that $\Phi({r}_M)=\sym(M^{tl})$.
\end{proof}
\begin{remark}
From this result, we see that $\Phi$ does not preserve images. It is easy to see that $\Phi$ preserves surjections, but as it does not preserve images it can not preserve injections. It could be interesting to consider derived functors of $\Phi$ to see if any new structures can be found.
\end{remark}

\begin{remark}
Another interesting approach, that we leave for the future, would be to consider the following generalization. Let $\fF$ be a coherent functor. Dualizing the natural map of Proposition~\ref{prop:naturalmap} applied to $\fF^\vee$, gives a natural map $\fF\to h^{\fF^\vee(A)}$. Analogously to the above, we then get a candidate for a natural definition of the Rees algebra of a coherent functor as 
\[\fR(\fF)=\im\Bigl(\Phi\bigl(\fF\bigr)\to\Phi\bigl(h^{\fF^\vee(A)}\bigr)\Bigr).\]
\end{remark}
\bibliographystyle{alpha}
\bibliography{references}{}
\end{document}